\documentclass[10pt,oneside,leqno]{amsart}
\usepackage{amsxtra}
\usepackage{amsopn}
\usepackage{color}
\usepackage{amsmath,amsthm,amssymb,tabularx}
\usepackage{amscd}
\usepackage{amsfonts}
\usepackage{latexsym}
\usepackage{verbatim}
\usepackage{graphicx}
\usepackage{graphics}
\usepackage{multirow}
\usepackage{lscape}
\usepackage{scrextend}
\usepackage{mathabx}
\usepackage{nicefrac}

\usepackage[hypertexnames=false,
 backref=UseNone,
    pdftex,
    pdfpagemode=UseNone,
    breaklinks=true,
    extension=pdf,
    colorlinks=true,
    linkcolor=blue,
    citecolor=blue,
    urlcolor=blue,
]
{hyperref}

\theoremstyle{plain}
\newtheorem{theorem}{Theorem}[section]

\newtheorem{lemma}[theorem]{Lemma}
\newtheorem{proposition}[theorem]{Proposition}
\newtheorem{corollary}[theorem]{Corollary}
\newtheorem{remark}[theorem]{Remark}

\newtheorem{example}[theorem]{Example}

\newtheorem{remark-question}[section]{Remark-Question}

\newcommand\frg{{\mathfrak g}}
\newcommand\frh{{\mathfrak h}}

\makeatletter
\renewcommand\scriptsize{\@setfontsize\scriptsize{8}{9}}
\makeatother

\definecolor{fondo}{rgb}{0.93,0.93,0.93}

\definecolor{m}{rgb}{0.9,0,0.9}

\makeatletter
\renewcommand*{\eqref}[1]{%
  \hyperref[{#1}]{\textup{\tagform@{\ref*{#1}}}}%
}
\makeatother

\begin{document}
\title[Laplacian coflow for warped $\mathrm{G}_2$-structures]{Laplacian coflow for warped $\mathrm{G}_2$-structures}

 \author{Victor Manero}
 \address[V. Manero]{Departamento de Matem\'aticas\,-\,I.U.M.A.\\
 Universidad de Zaragoza\\
 Facultad de Ciencias Humanas y de la Educaci\'on\\
 22003 Huesca, Spain}
 \email{vmanero@unizar.es}

 \author{Antonio Otal}
 \address[A. Otal and R. Villacampa]{Centro Universitario de la Defensa\,-\,I.U.M.A., Academia General
 Mili\-tar, Crta. de Huesca s/n. 50090 Zaragoza, Spain}
 \email{aotal@unizar.es}
 \email{raquelvg@unizar.es}

 \author{Raquel Villacampa}


\maketitle

\begin{abstract}
We consider the Laplacian coflow of a $\mathrm{G}_2$-structure on warped products of the form $M^7= M^6 \times_f S^1$ with $M^6$ a compact 6-manifold endowed with an $\mathrm{SU}(3)$-structure. We give an explicit reinterpretation of this flow as a set of evolution equations of the differential forms defining the $\mathrm{SU}(3)$-structure on $M^6$ and the warping function $f$. Necessary and sufficient conditions for the existence of solution for this flow are given. Finally we describe new long time solutions for this flow where the $\mathrm{SU}(3)$-structure on $M^6$ is nearly K\"ahler, symplectic half-flat or balanced.
\end{abstract}

\setcounter{tocdepth}{3} \tableofcontents

\bigskip

\section*{Introduction}

The first author to consider flows of $\mathrm{G}_2$-structures was Bryant in 2006, \cite{Br}. Concretely he considered the Laplacian flow of a $\mathrm{G}_2$-structure:
$$\frac{\partial}{\partial t} \varphi(t)= \Delta_7 \varphi(t),$$
starting from $\varphi_0$ a closed 3-form defining the G$_2$-structure. $\Delta_7$ is the corresponding Hodge Laplacian, given by $\Delta_7=\, \ast_7 \, d_7 \ast_7 d_7 - d_7 \ast_7 d_7\, \ast_7$. 

In the last years there has been a lot of fundamental works on this issue. In \cite{BX} it was proved the short time existence and uniqueness of solution on compact manifolds. The first examples of long time solutions to this flow were described in \cite{FFM}. These examples consist on non compact nilpotent Lie groups endowed with a one-parameter family of closed $\mathrm{G}_2$-structures such that satisfy the Laplacian flow equation for all $t \in (a, +\infty)$ with $a<0$.

Recent papers by Lotay and Wei \cite{LW1, LW2, LW3} derived important properties of the Laplacian flow as long time existence or convergence results. Even more recently Fino
and Raffero on \cite{FR} obtained sufficient conditions for the existence of solution of this flow on warped products of the form
$M^6 \times_f S^1$ with $M^6$ a 6-dimensional manifold endowed with an $\mathrm{SU}(3)$-structure. Recall that, if $(B,g_B)$ and $(F,g_F)$ are Riemannian
manifolds and $f$ is a non-vanishing differentiable function on $B$, then the warped product $W=B\times_f F$ consists on the product manifold
$B\times F$ endowed with the metric $g=\pi_1^*(g_B)+f^2 \pi_2^*(g_F)$ where $ \pi_1$ and $ \pi_2$ are the projections of $W$ onto $B$ and $F$
respectively. They also reinterpret the flow as a set of evolution equations on $M^6$ involving  the differential forms defining the
$\mathrm{SU}(3)$-structure and the warping function $f$. More details about the Laplacian flow of a closed $\mathrm{G}_2$-structure can be found in the reviews \cite{FFR, Lo} and the references therein. Another interesting result concerning this flow was due to Xu and Ye in \cite{XY}, where they proved long time existence and uniqueness of solution for this flow starting near a torsion free $\mathrm{G}_2$-structure.

In this work we consider the so-called Laplacian ``coflow" of $\mathrm{G}_2$-structures. This coflow was introduced by Karigiannis, McKay and Tsui in \cite{KMT} and can be considered as the analogous of the Laplacian flow of a closed $\mathrm{G}_2$-structure where the 3-form $\varphi_0$ is now considered to be coclosed instead of closed. Equivalently this flow can be stated as:
$$\frac{\partial}{\partial t} \ast_7 \varphi(t) = -\Delta_7 \ast_7 \varphi(t) ,$$
where the 4-form $\ast_7 \varphi_0$ is closed and $\ast_7$ denotes the Hodge star operator. These authors considered more natural to define this flow with a minus sign in order to make it more likely to the heat equation. In order to obtain solutions they consider 7-dimensional manifolds $M^6 \times L^1$ with $L^1=\mathbb{R}$ or $S^1$ where $M^6$ is endowed with a Calabi-Yau or a nearly K\"ahler structure. Grigorian in \cite{G} introduced the modified Laplacian coflow, which consists on a modified version of the Laplacian coflow, proving short time existence and uniqueness of solution  for this modified flow. He also derives the modified Laplacian coflow for warped $\mathrm{G}_2$-structures of the form $M^6 \times_f L^1$ obtaining solution for $M^6$ being Calabi-Yau or nearly K\"ahler. Long time solutions for the Laplacian coflow on non compact nilpotent Lie groups were described in \cite{BFF}.  In this work we present long time solutions for the coflow on warped products where the base manifolds are Lie groups endowed with metrics belonging to the Gray-Hervella classes $\mathcal W_1\oplus \mathcal W_2\oplus \mathcal W_3 $.

The paper is structured as follows. In Section 1 we give an introduction to $\mathrm{SU}(3)$ and $\mathrm{G}_2$-structures.
Section 2 is devoted to $\mathrm{G}_2$-structures of the form $M^6\times_f S^1$ ($M^6$ being compact and endowed with an SU(3)-structure) whose induced metric describes a warped product. In particular in Theorem 2.3 we give an explicit description of the torsion forms of such a $\mathrm{G}_2$-structure in terms of the torsion forms of the $\mathrm{SU}(3)$-structure on the base manifold and the warping function. In Section 3 we reinterpret the Laplacian flow and coflow of a $\mathrm{G}_2$-structure as a set of evolution equations of the $\mathrm{SU}(3)$-structure and  we describe the Laplacian coflow operator of the warped $\mathrm{G}_2$-structure by means of the torsion forms of the $\mathrm{SU}(3)$-structure and the warping function.  In particular for the Laplacian flow we reobtain the equations due to Fino and Raffero in \cite{FR}. Finally the goal of Section 4 is to obtain new examples of long time solutions of the Laplacian coflow constructed as warped products where the base manifolds are 6-dimensional  and they are endowed with nearly K\"ahler,  symplectic half-flat or balanced SU(3)-structures.

\section{$\mathrm{SU}(3)$ and $\mathrm{G}_2$-structures}

In this section we review some preliminaries concerning $\mathrm{SU}(3)$ and $\mathrm{G}_2$-structures. More concretely we present these structures, their corresponding $\mathrm{SU}(3)$ and $\mathrm{G}_2$ type decomposition of the space of differential forms and finally their torsion forms.

\subsection{$\mathrm{SU}(3)$-structures}

An $\mathrm{SU}(n)$-structure on a differentiable manifold $M^{2n}$ consists on a triple $(g, J, \Psi)$ where $(g, J)$ is an almost Hermitian structure on $M^{2n}$ and $\Psi$ is a complex $(n,0)$ form, satisfying

\begin{equation*}
(-1)^{n(n-1)/2}\Big(\frac{\imath}{2}\Big)^n \Psi \wedge \overline{\Psi}= \frac{1}{n!} \, \omega^n,
\end{equation*}
with $\overline{\Psi}$ the conjugated form of $\Psi$ and $\omega$ the K\"ahler form of the almost Hermitian structure. An $\mathrm{SU}(n)$-structure can equivalently be described by the triple $(\omega, \psi_{+}, \psi_{-})$ where $\psi_+$ and $\psi_-$ are, respectively the real and the imaginary part of the complex form $\Psi$.  In what follows we will focus on $\mathrm{SU}(3)$-structures on 6-dimensional manifolds. Note that in this case, the metric $g_{\omega,\psi_{\pm}}$ can be recovered from $(\omega, \psi_{+}, \psi_{-})$ as
$$g_{\omega,\psi_{\pm}}(X,Y) vol_6= -3 \, (\iota_X) \omega \wedge (\iota_Y \psi_+) \wedge \psi_+, $$
where $\iota$ denotes the contraction operator, $vol_{6}=\frac{1}{3!}\omega^3$  and $X, Y \in \mathfrak{X}(M^{2n})$. 

The presence of such a structure on a manifold $M^6$ can also be characterized by the existence of a local basis of 1-forms $\{e^1,\dots, e^6\}$ such that $(\omega, \psi_+, \psi_-)$ can be described as:

\begin{equation}\label{adaptedBasis}
\begin{array}{c}
\omega = e^{12}+e^{34}+e^{56},\\[10pt]
\psi_+ =e^{135}-e^{146}-e^{236}-e^{245},\quad 
\psi_- =-e^{246}+e^{235}+e^{145}+e^{136},
\end{array}
\end{equation}
where we denote, as usual in the related literature, $e^{ij}$ the wedge product $e^i \wedge e^j$ and $e^{ijk}$ the wedge product $e^{i} \wedge e^{j} \wedge e^{k}$. In the following, a basis in which the $\mathrm{SU}(3)$-structure has the expression \eqref{adaptedBasis} will be called an \emph{adapted basis}.

In \cite{CS} it is described how the intrinsic torsion of an $\mathrm{SU}(3)$-structure, namely $\tau$, lies in a space of the form
$$\tau \in \mathcal{W}_1^{\pm}\oplus \mathcal{W}_2^{\pm}\oplus \mathcal{W}_3\oplus \mathcal{W}_4\oplus \mathcal{W}_5, $$
where $\mathcal{W}_i$ denote the irreducible components under the action of the group $\mathrm{SU}(3)$. This torsion can be described by the exterior derivatives of $\omega, \psi_+$ and $\psi_-$ and also in terms of the so called torsion forms. This latter description is given in \cite{BV} where the authors consider the natural action of the group $\mathrm{SU}(3)$ on $\Omega^k(M^6)$, the space of $k$-forms on $M^6$. Thus, the different spaces of forms $\Omega^k(M^6)$ can be splitted into $\mathrm{SU}(3)$ irreducible subspaces as follows:

\medskip

\centerline{$\Omega^1(M^6)$ is irreducible,}

\medskip

\centerline{$\Omega^2(M^6)=\Omega^2_{1}(M^6) \oplus \Omega^2_{6}(M^6) \oplus \Omega^2_{8}(M^6) $, }

\medskip

with

\medskip

$\Omega^2_1(M^6)=\{ f \omega | f \in \mathcal{C}^{\infty}(M^6)\},$

\medskip

$\Omega^2_6(M^6)=\{ \ast_6 J(\eta \wedge \psi_+) | \eta \in \Omega^1(M^6) \}= \{\sigma \in \Omega^2(M^6) | J\sigma = \sigma\}$,

\medskip

$\Omega^2_8(M^6)= \{ \sigma \in \Omega^2(M^6) | \sigma \wedge \psi_+=0, \ast_6 J \sigma = -\sigma \wedge \omega \}= \{\sigma \in \Omega^2(M^6)|J\sigma = -\sigma, \sigma \wedge \omega^2=0\}$;

\medskip

and

\medskip

\centerline{$\Omega^3(M^6)=\Omega^3_{+}(M^6) \oplus \Omega^3_{-}(M^6) \oplus \Omega^3_{6}(M^6) \oplus \Omega^3_{12}(M^6)$, }

\medskip

with

\begin{equation*}
\begin{array}{ll}
\Omega^3_{+}(M^6) = \{f \psi_+ |\,  f \in \mathcal{C}^{\infty}(M^6)\},&\quad \Omega^3_{6}(M^6)=\{\eta \wedge \omega |\, \eta \in \Omega^1(M^6)\}=\{\gamma \in \Omega^3(M^6)| \ast_6 J \gamma=\gamma\},\\[7pt]
\Omega^3_{-}(M^6) = \{f \psi_- |\,  f \in \mathcal{C}^{\infty}(M^6)\},&\quad  \Omega^3_{12}(M^6) = \{\gamma \in \Omega^3(M^6) |\, \gamma \wedge \omega = 0, \gamma \wedge \psi_{\pm}=0\},
\end{array}
\end{equation*}

\noindent
where $\ast_6$ denotes the Hodge star operator associated to the induced metric $g_{\omega, \psi_{\pm}}$ and the volume form $vol_{6}$. Notice that $\Omega^k_d(M^6)$ denotes the $\mathrm{SU}(3)$-irreducible space of $k$-forms having dimension $d$. Decompositions  of the spaces of $k$-forms for $k=4,5$ and $6$ need not to be detailled since they can be achieved via the Hodge star operator, $\ast_6 \Omega^k_d(M^6)= \Omega^{6-k}_d(M^6)$.

With all these previous descriptions the derivatives of $\omega, \psi_+$ and $\psi_-$ can be decomposed into summands belonging to the $\mathrm{SU}(3)$-invariant spaces as follows (see \cite{BV} for details):
\begin{equation}\label{SU3-str}
\begin{array}{lll}
d\omega &=& \frac{-3}{2}\sigma_0\psi_+ + \frac{3}{2}\pi_0\psi_- +\nu_1\wedge \omega+ \nu_3,\\[5pt]
d\psi_+ &=&\pi_0\,\omega^2 + \pi_1\wedge \psi_+ - \pi_2\wedge \omega,\\[5pt]
d\psi_- &=&\sigma_0\,\omega^2 + \pi_1\wedge \psi_- - \sigma_2\wedge \omega,
\end{array}
\end{equation}
where $\sigma_0, \pi_0 \in \mathcal{C}^{\infty}(M^6), \pi_1, \nu_1 \in \Omega^1(M^6), \pi_2, \sigma_2 \in \Omega^2_{8}(M^6)$ and $\nu_3\in \Omega^3_{12}(M^6)$ are the \emph{torsion forms} of the $\mathrm{SU}(3)$-structure.
\bigskip

Some classes of $\mathrm{SU}(3)$-structures that are useful for our purposes are given in Table~\ref{tablaSU3}.

\begin{table}[h!]
\renewcommand{\arraystretch}{1.3}
\begin{center}
\begin{tabular}{|c|c|l|}
\hline
Class& Non-vanishing torsion forms & Structure \\
\hline
$\{0\}$& -- & Calabi-Yau \\
\hline
$\mathcal W_1^-$& $\sigma_0 $& Nearly K\"ahler \\
\hline
$\mathcal W_2^-$& $\sigma_2 $& Symplectic half-flat \\
\hline
$\mathcal W_3 $& $\nu_3 $& Balanced \\
\hline
\end{tabular}
\medskip
\caption{Some classes of $\mathrm{SU}(3)$-structures}
\label{tablaSU3}
\end{center}
\end{table}

\subsection{$\mathrm{G}_2$-structures}

A $\mathrm{G}_2$-structure on a 7-dimensional differentiable manifold consists on a three form $\varphi$ such that it defines a metric, namely $g_{\varphi}$, a volume form $vol_{7}$ and a 2-fold vector cross product, see \cite{FG,Hi}. 
The metric $g_{\varphi}$ can be recovered from $\varphi$ as
$$g_{\varphi}(X,Y) vol_{7}= \frac16 \, (\iota_X \varphi) \wedge (\iota_Y \varphi) \wedge \varphi, $$
with $X, Y \in \mathfrak{X}(M^{7})$. The presence of such structure on a manifold $M^7$ can be characterized by the existence of an adapted basis, i.e. a local basis of 1-forms $\{e^1,\dots, e^7\}$ such that $\varphi$ can be described as:
\begin{equation*}
\varphi =e^{127}+e^{347}+e^{567}+e^{135}-e^{146}-e^{236}-e^{245}.
\end{equation*}


Concerning the intrinsic torsion of a $\mathrm{G}_2$-structure, namely $\mathcal{T}$, in \cite{FG} it is described how this torsion lies in a space of the form
$$\mathcal{T}\in \mathcal{X}_1 \oplus \mathcal{X}_2 \oplus \mathcal{X}_3\oplus \mathcal{X}_4, $$
where $\mathcal{X}_i$ denotes the irreducible components under the action of the group $\mathrm{G}_2$. Thus, we can distinguish between 16 different classes of $\mathrm{G}_2$-structures, the so-called \emph{Fern\'andez-Gray classes}, which can be characterized by the behavior of the exterior derivative of $\varphi$ and $\ast_7 \varphi$ where $\ast_7$ is the Hodge star operator induced by the $\mathrm{G}_2$-structure. In \cite{Br} it is given a description of the derivatives of $\varphi$ and $\ast_7 \varphi$ as summands belonging to the different $\mathrm{G}_2$-invariant spaces $\mathcal{X}_i$.

To obtain this description it is considered the natural action of the group $\mathrm{G}_2$ on $\Omega^k(M^7)$.  
Thus, the different spaces of forms $\Omega^k(M^7)$ can be splitted into $\mathrm{G}_2$-irreducible subspaces as follows:

 \medskip

\centerline{$\Omega^1(M^7)$ is irreducible,}

\medskip

\centerline{$\Omega^2(M^7)=\Omega^2_{7}(M^7) \oplus \Omega^2_{14}(M^7) $, }

\medskip

with

\medskip

$\Omega^2_7(M^7)=\{ \ast_7(\eta \wedge \ast_7 \varphi)  |\, \eta \in \Omega^1(M^7)\}= \{ \sigma \in \Omega^2(M^7) | \sigma \wedge \varphi = 2 \ast_7 \sigma \}$,

\medskip

$\Omega^2_{14}(M^7)= \{ \sigma \in \Omega^2(M^7) |\, \sigma \wedge \varphi = - \ast_7 \sigma \}$;

\medskip

and

\medskip

\centerline{$\Omega^3(M^7)=\Omega^3_{1}(M^7) \oplus \Omega^3_{7}(M^7) \oplus \Omega^3_{27}(M^7) $, }

\medskip

with

\begin{equation*}
\begin{array}{ll}
\Omega^3_{1}(M^7) = \{f \varphi |\, f \in \mathcal{C}^{\infty}(M^7)\},&\quad \Omega^3_{27}(M^7)=\{  \gamma \in \Omega^3(M^7) |\, \gamma \wedge \varphi = \gamma \wedge \ast_7 \varphi=0\}.\\[10pt]
\Omega^3_{7}(M^7) = \{\ast_7(\eta \wedge \varphi) |\,  \eta \in \Omega^1(M^7)\},&
\end{array}
\end{equation*}

Similarly to the previous case, $\Omega^k_d(M^7)$ denotes the $\mathrm{G}_2$-irreducible space of $k$-forms which has dimension~$d$.  For the rest of dimensions ($k=4,5,6$ and $7$) use the relation: $\ast_7 \Omega^k_d(M^7)= \Omega^{7-k}_d(M^7)$.

Thus, the derivatives of $\varphi$ and $\ast_7 \varphi$ can be decomposed into summands belonging to the $\mathrm{G}_2$-invariant spaces as follows (see \cite{Br}):
\begin{equation}\label{tor}
d\,\varphi = \tau_0 \ast_7 \varphi + 3 \tau_1 \wedge \varphi + \ast_7 \tau_3,\qquad 
d(\ast_7\varphi) =4 \tau_1 \ast_7 \varphi +  \tau_2 \wedge \varphi,
\end{equation}
where $\tau_0 \in \mathcal{C}^{\infty}(M^7), \tau_1  \in \Omega^1(M^7), \tau_2 \in \Omega^2_{14}(M^7)$ and $\tau_3 \in \Omega^3_{27}(M^7)$ are the torsion forms.
\bigskip

In particular:
\begin{equation}\label{TF-7d}
\begin{array}{ll}
\tau_0= \dfrac17\ast_7(d\varphi \wedge\varphi),\quad & \qquad 
\tau_2 = -\ast_7 d\ast_7\varphi + 4\ast_7(\tau_1\wedge \ast_7\varphi),\\[5pt]
\tau_1= \dfrac{-1}{12}\ast_7(\ast_7 d\varphi \wedge\varphi),\quad &\qquad 
\tau_3 = \ast_7 d\varphi - \tau_0\varphi - 3\ast_7(\tau_1\wedge\varphi).\\
\end{array}
\end{equation}
%

The principal Fern\'andez-Gray classes are given in Table~\ref{tablaG2}:

\begin{table}[h!]
\renewcommand{\arraystretch}{1.3}
\begin{center}
\begin{tabular}{|c|c|l|}
\hline
Class& Non-vanishing torsion forms & Structure \\
\hline
$\mathcal P$& $-$& Parallel\\
\hline
$\mathcal X_1$& $\tau_0$& Nearly Parallel\\
\hline
$\mathcal X_2$& $\tau_2$& Closed\\
\hline
$\mathcal X_3$& $\tau_3$& Coclosed of pure type\\
\hline
$\mathcal X_4$& $\tau_1$& Locally conformal parallel\\
\hline
$\mathcal X_1 \oplus \mathcal X_3$& $\tau_0, \tau_3$& Coclosed\\
\hline
\end{tabular}
\medskip
\caption{Some classes of $\mathrm{G}_2$-structures}
\label{tablaG2}
\end{center}
\end{table}


\section{Warped $\mathrm{G}_2$-structures}

Consider two Riemannian manifolds, namely $(F,g_F)$ and $(B, g_B)$, and $f$ a non-vanishing real differentiable function on $B$. The warped product, denoted as $B \times_f F$, consists on the product manifold $$W=B \times F$$ endowed with the metric $g_f=\pi_1^*(g_B)+f^2 \pi_2^*(g_F)$
with $ \pi_1$ and $ \pi_2$ being the projections of $W$ onto $B$ and $F$ respectively.

Starting from an $\mathrm{SU}(3)$-structure $(\omega, \psi_{\pm})$ over $M^6$, and considering a function $f \in \mathcal{C}^{\infty}(M^6)$ it is possible to construct a $\mathrm{G}_2$-structure $\varphi$ over $M^7 = M^6\times S^1$ such that:
\begin{equation}\label{phi-warped}
\varphi = f\,\omega\wedge ds + (\alpha\,\psi_+ - \beta\,\psi_-),
\end{equation}
 with $s$ the coordinate on $S^1$ and $\alpha,\, \beta\in\mathbb R$ satisfying $\alpha^2+\beta^2=1$. Thus, the metric and the volume form of this $\mathrm{G}_2$-structure are given in terms of the $\mathrm{SU}(3)$-structure by:
$$g_{\varphi}= g_{\omega, \psi_{\pm}}+f^2 ds^2,\qquad vol_7=f vol_6 \wedge ds.$$ Observe that $g_{\varphi} = g_f$, so $M^7$ is in fact a warped product.
In what follows we will call \emph{warped $\mathrm{G}_2$-structure} to this $\mathrm{G}_2$-structure~\eqref{phi-warped}.
\begin{remark}
If we consider the pair $(\alpha,\beta)=(1,0)$, this definition of warped $\mathrm{G}_2$-structure is exactly the one already given in \cite{FR}.
\end{remark}
The metrics $g_{\omega,\psi_{\pm}}$ and $ g_\varphi$ on the base manifold $M^6$ and the warped product $M^6\times_f S^1$ respectively define two star operators $\ast_6$ and $\ast_7$ related by the following:
\begin{lemma}[Lemma 3.2, \cite{FR}]\label{lemma}
Let  $\eta \in \Omega^k(M^6)$ be a differential $k$-form on $M^6$, and let $\ast_6$ and $\ast_7$ be the Hodge star operator determined by the $\mathrm{SU}(3)$-structure and the warped $\mathrm{G}_2$-structure, respectively. Then
\begin{eqnarray*}
\ast_7 \eta &=& f \ast_6 \eta \wedge ds,\\[5pt]
\ast_7 (\eta \wedge ds)& =&(-1)^k f^{-1} \ast_6 \eta.
\end{eqnarray*}
\end{lemma}
Hence from \eqref{phi-warped} and the previous lemma it can be checked that
 \begin{equation}\label{star-phi-warped}
 \ast_7 \varphi=  \frac12 \omega^2 + f\,(\alpha\,\psi_- + \beta\,\psi_+)\wedge ds.
 \end{equation}

\begin{remark}
The key idea of this section is to study how the $\mathrm{G}_2$-geometry of the warped product $M^6\times_fS^1$ forces
conditions on the $\mathrm{SU}(3)$-geometry of the base $M^6$.  Having this idea in mind, we are going to describe the torsion forms~\eqref{TF-7d} of the warped $\mathrm{G}_2$-structure in terms of the torsion forms of the $\mathrm{SU}(3)$-structure and the warping function.
\end{remark}

In the spirit of \cite[Theorem 3.4]{MU} we can prove:

\begin{theorem}
Let $(M^6, \omega, \psi_{\pm})$ be an $\mathrm{SU}(3)$-manifold with torsion forms $\pi_0, \sigma_0, \pi_1, \nu_1, \pi_2, \sigma_2$ and $\nu_3$.  Then, the torsion forms~\eqref{TF-7d} of a warped $\mathrm{G}_2$-manifold $(M^7 = M^6\times_f S^1,\,\varphi)$ are given by
\begin{equation}\label{torsiones}
\begin{array}{lll}
\tau_0 &=& \frac{12}{7} (\alpha\pi_0-\beta\sigma_0),\\
&& \\
\tau_1 &=&  
\frac{1}{2} (\alpha\sigma_0+\beta\pi_0) fds + \frac16 \eta_1 ,\\
&& \\
\tau_2 
  &=& -\alpha\sigma_2-\beta\pi_2+\frac{f}{3}\ast_6\left(\eta_2\wedge\omega^2\right)\wedge ds -\frac13\ast_6\left(\eta_2\wedge(\alpha\psi_-+\beta\psi+)\right),\\
  && \\
\tau_3&=& 
\left[\frac27(\alpha\pi_0-\beta\sigma_0)f\omega-\frac{f}{2}\ast_6\left(\eta_3\wedge(\alpha\psi_+-\beta\psi_-)\right)+f(\alpha\pi_2-\beta\sigma_2)\right]\wedge ds-\frac12\ast_6(\eta_3\wedge\omega)-\\[1em]
&&\frac{3}{14}(\alpha\pi_0-\beta\sigma_0)(\alpha\psi_+-\beta\psi_-)-\ast_6\nu_3,
\end{array}
\end{equation}
where $\eta_i$ are the following 1-forms: $$\eta_1 = \frac{1}{f} d_6 f + \pi_1 +\nu_1,\quad \eta_2 = \frac{1}{f} d_6f+\pi_1-2\nu_1,\quad \eta_3 = \frac{1}{f} d_6f- \pi_1+\nu_1.$$
\end{theorem}

\begin{proof}
The result holds after long computations where the definition of the spaces $\Omega^k_d(M^6)$ are used.  As hint, let us write down  the expressions for $d\varphi,\, \ast_7(d\varphi)$ and $d(\ast_7\varphi)$.  From~\eqref{phi-warped} and~\eqref{star-phi-warped} one gets:
\begin{eqnarray*}d\varphi &=& \left(df\wedge \omega - \frac32 f \sigma_0 \psi_+ + \frac32 f \pi_0 \psi_- + f\,\nu_1\wedge \omega + f\,\nu_3\right)\wedge ds\\
&&+ (\alpha\,\pi_0 - \beta\,\sigma_0) \omega^2 + \pi_1\wedge (\alpha \psi_+ - \beta \psi_-) - (\alpha\,\pi_2 - \beta\,\sigma_2)\wedge \omega,
\end{eqnarray*}
\begin{eqnarray*}
\ast_7 (d\varphi) &=&   -f^{-1} \ast_6 (df\wedge \omega) + \frac32 \sigma_0 \psi_- + \frac32  \pi_0 \psi_+ - \ast_6 (\nu_1\wedge \omega) - \ast_6\,\nu_3\\
&& + \left[2f (\alpha\,\pi_0 - \beta\,\sigma_0) \omega +f \ast_6 (\pi_1\wedge (\alpha \psi_+ - \beta \psi_-)) +\alpha\, f\,\pi_2 - \beta\,f\,\sigma_2\right]\wedge ds,
\end{eqnarray*}
\begin{eqnarray*}
d (\ast_7\varphi) &=& \nu_1\wedge \omega^2+ \left[-f(\alpha\sigma_2 + \beta\pi_2)\wedge\omega + f(\alpha\sigma_0 + \beta\pi_0)\omega^2 + (df + f\pi_1)\wedge (\alpha\psi_- + \beta\psi_+)\right]\wedge ds.
\end{eqnarray*}

Finally, from \eqref{TF-7d} and using Lemma \ref{lemma} the result is achieved after long and standard computations.

\end{proof}

Most of the Fern\'andez-Gray classes of $\mathrm{G}_2$-structures are characterized in terms of the cancellation of some of their
torsion forms (see Table~\ref{tablaG2}). Using expressions \eqref{torsiones}, the cancellations of $\tau_0,\,\tau_1,\,\tau_2$ and $\tau_3$ are expressed by using
the $\mathrm{SU}(3)$-torsion forms of the base $M^6$ and the warping function $f$.

\begin{corollary}\label{corolary1}
Let $(M^6, \omega, \psi_{\pm})$ be an $\mathrm{SU}(3)$-manifold. Thus, the torsion forms of the warped $\mathrm{G}_2$-structure satisfy:
\begin{equation*}
\begin{aligned}
\tau_0   = 0  & \iff \left \{  \begin{array}{ll} i) \, \, \, \, \,  \, \, & \alpha\pi_0-\beta\sigma_0 =0.
\end{array}
\right.
\\
\tau_1  = 0  & \iff    \left \{  \begin{array}{ll} ii)   \, \, \, & \alpha\sigma_0+\beta\pi_0=0,\\[3pt]
                                       iii)  \,  \, \, &\eta_1 = 0.
\end{array}
\right.
\\
\tau_2  = 0  & \iff \left \{  \begin{array}{ll} iv)  \, \, \, \, & \eta_2 = 0,\\[3pt]
                                       v)  \,  \, \, \, & \alpha\sigma_2+\beta\pi_2 =0.
\end{array}
\right.
  \\
 \tau_3 = 0 & \iff \left \{  \begin{array}{ll} vi)   &  \alpha\pi_0-\beta\sigma_0 =0, \\[3pt]
 vii)   & \eta_3 = 0,\\[3pt]
 viii) &  \alpha\pi_2-\beta\sigma_2 =0,\\[3pt]
 ix)  & \nu_3=0,
\end{array}
\right.
 \\
\end{aligned}
\end{equation*}
\end{corollary}

In Table~\ref{tabla1} we show how the $\mathrm{G}_2$-geometry of the warped product $M^6\times_fS^1$ forces conditions on the $\mathrm{SU}(3)$-geometry of the base $M^6$.

\scriptsize{
\begin{table}[h!]
\renewcommand{\arraystretch}{1.8}
\begin{center}
\begin{tabular}{|c|c||c|c|}
\hline
Class& $\mathrm{G}_2$-torsion forms & $\mathrm{SU}(3)$-torsion forms&Class\\
\hline
\multirow{2}{*}{$\mathcal P$}& \multirow{2}{*}{$\tau_0 = \tau_1 = \tau_2 = \tau_3 = 0$}& $\sigma_i = \pi_i=\nu_i=0$& \multirow{2}{*}{${0}$} \\
& &$d_6f = 0$&\\
\hline
\multirow{3}{*}{$\mathcal X_2$}& \multirow{3}{*}{$\tau_0 = \tau_1 =\tau_3 = 0$}& $\pi_0=\sigma_0=\pi_1 = \nu_3 =0$&\multirow{3}{*}{$\mathcal W_2^{\pm}\oplus \mathcal W_4$}\\
& &$\alpha\,\pi_2-\beta\sigma_2=0$&\\
& &$\frac1f d_6 f = -\nu_1$&\\
\hline
\multirow{3}{*}{$\mathcal X_3$}& \multirow{3}{*}{$\tau_0 = \tau_1 = \tau_2= 0$}&$\pi_0 = \sigma_0 = \nu_1 = 0$&\multirow{3}{*}{$\mathcal W_2^{\pm}\oplus \mathcal W_3\oplus \mathcal W_5$}\\
& &$\alpha \sigma_2 + \beta \pi_2 = 0$&\\
& &$\frac1f d_6 f = -\pi_1$&\\
\hline
\multirow{2}{*}{$\mathcal X_4$}& \multirow{2}{*}{$\tau_0 = \tau_2 = \tau_3 = 0$}&$\sigma_2 = \pi_2 = \nu_3=0$&\multirow{2}{*}{$\mathcal W_1^{\pm}\oplus \mathcal W_4\oplus \mathcal W_5$}\\
& &$\frac1f d_6 f = \frac12 \nu_1 = \frac13 \pi_1$&\\
\hline
\multirow{3}{*}{$\mathcal X_1 \oplus \mathcal X_3$}& \multirow{3}{*}{$\tau_1 = \tau_2 = 0$}& $\alpha\sigma_0 + \beta\pi_0=0$&\multirow{3}{*}{$\mathcal W_1^{\pm}\oplus\mathcal W_2^{\pm}\oplus \mathcal W_3\oplus \mathcal W_5$}\\
& &$\alpha\sigma_2 + \beta\pi_2=0$&\\
& &$\nu_1=0,\,\, \frac1f d_6 f = -\pi_1$&\\
\hline
\end{tabular}
\medskip
\caption{Relation between torsion forms of the warped $\mathrm{G}_2$-structure and the $\mathrm{SU}(3)$-structure}
\label{tabla1}
\end{center}
\end{table}
}

\normalsize{
\begin{remark}
From Corollary \ref{corolary1}, $\tau_3=0$ implies $\tau_0=0$, therefore nearly Parallel structures can not be achieved as warped $\mathrm{G}_2$-structures of the form~\eqref{phi-warped}.
\end{remark}
}

\section{ The Laplacian flow and coflow of warped $\mathrm{G}_2$-structure of the form $M^6\times_f S^1$ }

Recall the definition of the Laplacian flow and coflow, that are respectively:
$$(LF)\,\,\begin{cases}
\dfrac{\partial}{\partial t} \varphi(t) = \Delta_t \varphi(t),\\[5pt]
d_7\,\varphi(t) = 0,
\end{cases}\qquad (LcF)\,\,\begin{cases}
\dfrac{\partial}{\partial t} (\ast_t\varphi(t)) = -\Delta_t (\ast_t\varphi(t)),\\[5pt]
d_7\,(\ast_t\varphi(t)) = 0,
\end{cases}$$
where $\varphi(t)$ is a one-parameter family of $\mathrm{G}_2$-structures and $\Delta_t$, $\ast_t$ denote the Laplacian and the Hodge star operator induced by $\varphi(t)$ for every $t$.

Our objective in this section is to particularize the Laplacian flow and coflow considering one-parameter families of $\mathrm{G}_2$-structures obtained as warped products, i.e.
\begin{equation}\label{varphi-t}
\varphi(t) = f(t) \omega(t)\wedge ds + (\alpha \psi_+(t) +  \beta\psi_-(t)).
\end{equation}
 From the previous expression, we derive the following: 
\begin{equation}\label{partials}
\begin{array}{l}\dfrac{\partial}{\partial t} \varphi(t)= \left(\dfrac{\partial }{\partial t}f(t)\,\omega(t) + f(t) \dfrac{\partial}{\partial t} \omega(t)\right) \wedge ds +\alpha \dfrac{\partial}{\partial t}\psi_+(t) -\beta \dfrac{\partial}{\partial t}\psi_-(t),\\[7pt]
\dfrac{\partial}{\partial t} (\ast_7\varphi(t))= \left[\dfrac{\partial }{\partial t}f(t)\,(\beta \psi_+(t) + \alpha \psi_-(t) )+ f(t) \left(\beta \dfrac{\partial}{\partial t} \psi_+(t) + \alpha \dfrac{\partial}{\partial t} \psi_-(t)\right)\right] \wedge ds +\dfrac12 \dfrac{\partial}{\partial t}\omega^2(t).
\end{array}
\end{equation}

Now we focus on the 3-form $\Delta_7\varphi$, resp. the 4-form $\Delta_7\ast_7\varphi$. For a generic G$_2$-structure, considering the formulas given in \eqref{tor} of the exterior derivatives of $\varphi$ and $\ast_7 \varphi$, a description of the Laplacian in terms of the torsion forms can be given as
\begin{equation}\label{Laplacian}
\Delta_7 \varphi=d_7(\tau_2-4\ast_7(\tau_1\wedge \ast_7\varphi))+\ast_7 d_7(\tau_0\varphi+3\ast_7(\tau_1 \wedge \varphi)+\tau_3).
\end{equation}
Since the Laplacian commutes with the Hodge star operator, $\Delta_7 \ast_7 = \ast_7 \Delta_7,$
%
%
%
combining \eqref{torsiones} and \eqref{Laplacian} it is also posible to describe $\Delta_7 \ast_7\varphi$ of a warped $\mathrm{G}_2$-structure in terms of the torsion forms of the $\mathrm{SU}(3)$-structure and the warping function $f$ for particular classes of $\mathrm{G}_2$-structures.

Provided that we are interested in the Laplacian flow, resp. coflow, we consider the 3-form $\Delta_7\varphi$, resp. the 4-form $\Delta_7\ast_7\varphi$, when $\varphi$ is closed, resp. coclosed. Let us start with the closed ones:

\begin{proposition}\label{prop-closed}
Let $\varphi$ be a warped closed $\mathrm{G}_2$-structure~\eqref{phi-warped} on $M^6\times_f S^1$ where $(\omega,\psi_{\pm})$ is an $\mathrm{SU}(3)$-structure on $M^6$.
Then $\Delta_7\varphi$ has the following expression:
 $$
\Delta_7\varphi=-d_6(\alpha \sigma_2+\beta \pi_2)+d_6\ast_6\big(\nu_1 \wedge (\alpha \psi_- + \beta \psi_+)\big)+f\big[\nu_1\wedge \ast_6(\nu_1\wedge \omega^2)- d_6 \ast_6 (\nu_1 \wedge \omega^2)\big]\wedge ds,
$$ where $\alpha\pi_2 - \beta \sigma_2 = 0$.

In the particular case that the warping function $f$ is constant $(d_6f=0)$, then
\begin{equation*}
\Delta_7\varphi=-d_6(\alpha \sigma_2+\beta \pi_2).
\end{equation*}

\begin{proof}
Since $\varphi$ is closed, $\tau_0=\tau_1=\tau_3=0$  and by \eqref{Laplacian} $$\Delta_7 \varphi = d_7 \tau_2,$$ where in view of~\eqref{torsiones}
$$\tau_2=-\alpha\sigma_2-\beta\pi_2+\ast_6(\nu_1\wedge(\alpha\psi_-+\beta\psi_+))-f\ast_6(\nu_1\wedge\omega^2)\wedge ds.$$
For the case $f$ constant, since $\frac{1}{f}d_6f=-\nu_1$ (see Table~ \ref{tabla1}) then  $\nu_1 = 0$ and  the result holds.
\end{proof}
\end{proposition}

Consider now coclosed $\mathrm{G}_2$-structures:
\begin{proposition}\label{prop-coclosed}
Let $\varphi$ be a warped coclosed $\mathrm{G}_2$-structure~\eqref{phi-warped} on $M^6\times_f S^1$ where $(\omega,\psi_{\pm})$ is an $\mathrm{SU}(3)$-structure on $M^6$.  Then $\Delta_7\ast_7\varphi$ has the following expression:

$$
\begin{array}{lll}
\Delta_7\ast_7\varphi&=&\frac32(\alpha\pi_0-\beta\sigma_0)\big[(\alpha\pi_0-\beta\sigma_0)\omega^2+\pi_1\wedge (\alpha\psi_+-\beta\psi_-)-(\alpha\pi_2-\beta\sigma_2)\wedge\omega\big]+d_6\ast_6(\pi_1\wedge \omega)\\[1em]
&&-d_6(\ast_6\nu_3)+\frac32 d_6(\alpha\pi_0-\beta\sigma_0)\wedge (\alpha \psi_+ - \beta \psi_- )\\[1em]
&&+f\Big[2d_6(\alpha\pi_0-\beta\sigma_0)\wedge \omega+(\alpha\pi_0-\beta\sigma_0)\left(-2\pi_1\wedge\omega-3\sigma_0\psi_++3\pi_0\psi_-+2\nu_3\right) +d_6(\alpha\pi_2-\beta\sigma_2)\\[1em]
&&-\pi_1 \wedge \ast_6 \big(\pi_1\wedge(\alpha \psi_+-\beta \psi_-)+d_6 \ast_6 \big(\pi_1\wedge(\alpha \psi_+-\beta \psi_-)\big)-\pi_1\wedge (\alpha \pi_2-\beta\sigma_2)\Big]\wedge ds,
\end{array}
$$ where $\alpha\sigma_i + \beta\pi_i =0$ for $i=0,2$.

Moreover, if $f$ is constant, then
\begin{equation}\label{laplacian-coclosed-fcte}
\begin{array}{lll}
\Delta_7\ast_7\varphi&=&\frac32(\alpha\pi_0-\beta\sigma_0)\big((\alpha\pi_0-\beta\sigma_0)\omega^2-(\alpha\pi_2-\beta\sigma_2)\wedge\omega\big)-d_6(\ast_6\nu_3) \\[1em]
&&+\frac32 d_6(\alpha \pi_0-\beta \sigma_0) \wedge (\alpha \psi_+-\beta \psi_-)\\[1em]
&&+f\Big[2d_6(\alpha\pi_0-\beta\sigma_0)\wedge \omega +(\alpha\pi_0-\beta\sigma_0)\left(-3\sigma_0\psi_++3\pi_0\psi_-+2\nu_3\right)+d_6(\alpha\pi_2-\beta\sigma_2)\Big]\wedge ds.\\[1em]
\end{array}
\end{equation}


\end{proposition}
\begin{proof}

The condition $\varphi$ being coclosed is equivalent to $\tau_1=\tau_2=0$ and as a consequence of \eqref{Laplacian}:
$$\Delta_7 \ast_7\varphi=\ast_7 \Delta_7 \varphi=d_7(\tau_0\varphi+\tau_3).$$
Now, using~\eqref{torsiones}:
$$
\begin{array}{lll}
\Delta_7 \ast_7\varphi&=&d_7 \Big[f\Big(2(\alpha \pi_0 -\beta \sigma_0)\omega +\ast_6\big(\pi_1\wedge(\alpha \psi_+-\beta \psi_-)\big)+ (\alpha \pi_2 - \beta \sigma_2) \Big)\wedge ds\\[1em]
&&+\frac32 (\alpha\pi_0-\beta\sigma_0)(\alpha \psi_+-\beta\psi_-)+\ast_6(\pi_1\wedge \omega)-\ast_6\nu_3\Big],
\end{array}
$$
and the result follows.  In order to prove~\eqref{laplacian-coclosed-fcte}, observe that $\pi_1 = 0$ according to Table \ref{tabla1}.



\end{proof}

\begin{remark}
  In what follows, and similarly as in \cite{FR}, we restrict our attention to the case when the warping function $f$ is constant over
  the base manifold $M^6$.

\end{remark}

In order to obtain solutions of the Laplacian flow of a warped closed $\mathrm{G}_2$-structure,  combining  the expressions \eqref{partials} and Proposition~\ref{prop-closed}, we can set the system of equations that must be satisfied:

\begin{proposition}
For a closed warped $\mathrm{G}_2$-structure \eqref{phi-warped}, the equation of the Laplacian flow (LF) is equivalent to:
\begin{equation*}\label{closed-flow}
\begin{cases}
f'(t) \, \omega(t) + f(t) \dfrac{\partial}{\partial t} \omega(t)=0,\\[5pt]
\alpha\dfrac{\partial}{\partial t}\psi_+(t) - \beta\dfrac{\partial}{\partial t}\psi_-(t) = -d_6(\alpha \sigma_2(t)+\beta \pi_2(t)).
\end{cases}
\end{equation*} where $\alpha \pi_2(t) - \beta \sigma_2(t) = 0$.
\end{proposition}

\begin{remark}
For the particular case of $(\alpha,\beta)=(1,0)$, we recover the system already studied by Fino and Raffero in \cite[Prop. 5.2]{FR}.
\end{remark}

\medskip

Similarly, for the coflow, we get the following system of equations:

\begin{proposition}
For a coclosed warped $\mathrm{G}_2$-structure \eqref{phi-warped}, the equation of the Laplacian coflow (LcF) is equivalent to:
\begin{equation*}
 \begin{cases}
\begin{array}{l}
 \dfrac{\partial \omega^2(t)}{\partial t} = -3(\alpha\pi_0(t)-\beta\sigma_0(t))^2\omega^2(t)+3(\alpha\pi_0(t)-\beta\sigma_0(t))(\alpha\pi_2(t)-\beta\sigma_2(t))\wedge \omega(t)\\[5pt] \hspace{3cm} +2d_6(\ast_6\nu_3(t))-3d_6(\alpha \pi_0(t) -\beta \sigma_0(t))\wedge (\alpha \psi_+(t)-\beta\psi_-(t)),\\[5pt]
\dfrac{f'(t)}{f(t)} \left(\beta\psi_+(t)+\alpha\psi_-(t)\right)+\left(\beta \dfrac{\partial \psi_+(t)}{\partial t}+\alpha \dfrac{\partial \psi_-(t)}{\partial t}\right) =  \\[5pt] \hspace{3cm} -(\alpha\pi_0(t)-\beta\sigma_0(t))\left[-3\sigma_0(t)\psi_+(t)+3\pi_0(t)\psi_-(t)+2\nu_3(t)\right]\\[5pt] \hspace{3cm} -d_6(\alpha \pi_2(t)-\beta \sigma_2(t)) -2d_6(\alpha \pi_0(t)-\beta \sigma_0(t))\wedge \omega(t),
 \end{array}
 \end{cases}
 \end{equation*} where $\alpha \sigma_i(t) + \beta \pi_i(t) = 0$ for $i=0,2$.
\end{proposition}

\begin{corollary}
For the particular case of $(\alpha,\beta)=(0,1)$, the Laplacian coflow becomes:

\begin{equation}\label{equacion}
 \begin{cases}
\begin{array}{l}
 \dfrac{\partial \omega^2(t)}{\partial t} = -3\sigma_0(t)^2\omega^2(t)+3 \sigma_0(t) \sigma_2(t) \wedge \omega(t)+2d_6(\ast_6\nu_3(t))-3d_6\sigma_0(t)\wedge\psi_-(t),\\[7pt]
\dfrac{f'(t)}{f(t)} \psi_+(t)+ \dfrac{\partial \psi_+(t)}{\partial t} =  -3\sigma_0(t)^2\psi_+(t)+2\sigma_0(t)\nu_3(t)+d_6\sigma_2(t)+2d_6\sigma_0(t)\wedge \omega(t).
 \end{array}
 \end{cases}
 \end{equation}
\end{corollary}

\medskip

\begin{remark}
For the Laplacian coflow we chose the parameters $(\alpha,\beta)$ to be $(0,1)$ in order to obtain equations depending on the torsion forms $\sigma_0, \sigma_2$ and $\nu_3$ (see~\eqref{SU3-str}) which are the ones that appear in the canonical definitions of the $\mathrm{SU}(3)$-structures, nearly K\"ahler, symplectic half-flat and balanced, respectively (see equations~\eqref{su3-nK}, \eqref{su3-sH} and \eqref{su3-bA} in the next sections).
\end{remark}

\section{New solutions to the Laplacian coflow}

Our main objective is to provide new long-time solutions $\varphi(t)$ for the Laplacian coflow~\eqref{equacion}.  In what follows we will consider one parameter families of warped G$_2$-structures~\eqref{varphi-t} on $G\times S^1$, being $G$ a Lie group.  The underlying $\mathrm{SU}(3)$-structures  $(\omega(t),\psi_+(t),\psi_-(t))$ are left-invariant and can be locally described
as 
\begin{equation}\label{SU3-x}
\begin{array}{c}
\omega(t) =x^{12}+x^{34}+x^{56},\\[10pt]
\psi_+(t) =x^{135}-x^{146}-x^{236}-x^{245},\quad \psi_-(t) =-x^{246}+x^{235}+x^{145}+x^{136},
\end{array}
\end{equation}
where $\{ x^i(t)\}$ denotes for every $t$ a local adapted basis and $x^{ij}$ stands for $x^i(t)\wedge x^{j}(t)$ and $x^{ijk}$ stands for $x^i(t)\wedge x^{j}(t) \wedge x^k(t)$.
Our ansatz consists on stating that 
\begin{equation}\label{evolution}
x^i(t) = f_i (t) h^i,
\end{equation}
 where $f_i(t)$ are differentiable non-vanishing real functions satisfying $f_i(0)=1$ and $\{h^1,\ldots, h^6\}$ is an adapted basis for the SU(3)-structure for $t=0$.  Notice that~\eqref{evolution} defines in fact a global basis since we are considering parallelizable manifolds.

Direct computations show:
\begin{eqnarray}\label{dw}
\frac{\partial \omega(t)}{\partial t}& = & \sum_{k=1}^3 \left(\frac{f_{2k-1}'(t)}{f_{2k-1}(t)} + \frac{f_{2k}'(t)}{f_{2k}(t)}\right)  x^{2k-1}(t)\wedge x^{2k}(t).
\end{eqnarray}

\medskip

\begin{equation}\label{dw2}
\begin{array}{lll}
  \dfrac{\partial \omega^2(t)}{\partial t} &=&2\sum_{(i,j,k,l)\in\mathcal J}\left(\dfrac{f_i'(t)}{f_i(t)} +\dfrac{f_j'(t)}{f_j(t)} +\dfrac{f_k'(t)}{f_k(t)} + \dfrac{f_l'(t)}{f_l(t)}  \right) x^{ijkl},\end{array}
\end{equation}

with $\mathcal{J}=\{(1,2,3,4),(1,2,5,6),(3,4,5,6)\}$.
\bigskip

\begin{equation}\label{segunda}
\begin{array}{lll}
\dfrac{f'(t)}{f(t)}\psi_+(t) + \dfrac{\partial \psi_+(t)}{\partial t}
&=&\left(\dfrac{f'(t)}{f(t)} + \dfrac{f_1'(t)}{f_1(t)} + \dfrac{f_3'(t)}{f_3(t)} + \dfrac{f_5'(t)}{f_5(t)} \right) x^{135}\\[10pt]&&-\sum_{(i,j,k)\in \mathcal{I}} \left( \dfrac{f'(t)}{f(t)} +\dfrac{f_i'(t)}{f_i(t)} + \dfrac{f_j'(t)}{f_j(t)} + \dfrac{f_k'(t)}{f_k(t)} \right)x^{ijk},
\end{array}
\end{equation}

with $\mathcal{I}=\{(1,4,6),(2,3,6),(2,4,5)\}$.

\medskip

As we mentioned before, the $\mathrm{G}_2$-geometry of the warped product imposes conditions on the $\mathrm{SU}(3)$-geometry
of the base $M^6$. Concretely, the $\mathrm{G}_2$-structure is coclosed if and only if the corresponding $\mathrm{SU}(3)$-structure lies on the space $\mathcal W_1^{\pm}\oplus\mathcal W_2^{\pm}\oplus \mathcal W_3\oplus \mathcal W_5$ (see Table~\ref{tabla1}). Notice that if we consider a one-parameter family of SU(3)-structures $(\omega(t),\, \psi_{\pm}(t))$ belonging to the previous space for any $t$, then the corresponding warped G$_2$-structure will remain coclosed for any $t$.
Now we particularize~\eqref{equacion}
for some interesting cases of $\mathrm{SU}(3)$-structures  lying on this space.

\subsection{The nearly K\"ahler case}

Recall that a nearly K\"ahler SU(3)-structure satisfies 
\begin{equation}\label{su3-nK}
d\omega= -\frac32\,\sigma_0\,\psi_+,\quad d\psi_+=0,\quad d\psi_- = \sigma_0 \omega^2.
\end{equation}
In particular, $\sigma_2=\nu_3=0$.
Particularizing~\eqref{equacion} for  $\sigma_2(t)=\nu_3(t)=0,$ we get

\begin{equation*}\label{equacion-NK}
\begin{cases}
\begin{array}{l}
\dfrac{\partial \omega^2(t)}{\partial t} = -3\sigma_0(t)^2\omega^2(t)-3d_6\sigma_0(t)\wedge\psi_-(t),\\[7pt]
\dfrac{f'(t)}{f(t)} \psi_+(t)+ \dfrac{\partial \psi_+(t)}{\partial t} =  -3\sigma_0(t)^2\psi_+(t)+2d_6\sigma_0(t)\wedge \omega(t).
\end{array}
\end{cases}
\end{equation*}

Observe that with this particular ansatz, the left-hand side of the first equation above  is a combination of the 4-forms $x^{1234}, x^{1256}$ and $x^{3456}$ (see~\eqref{dw2}); however, it can be easily proven that if $\eta$ is a one-form, then $\eta\wedge \psi_-(t)$ never belongs to the space generated by $x^{1234}, x^{1256}$ and $x^{3456}$, unless $\eta=0$.  Therefore, we need $d_6\sigma_0(t)=0$, which means that $\sigma_0(t)$ is constant as a differentiable function on $M^6$.    

Now, the previous system simplifies as:
\begin{equation}\label{equacion-NK2}
 \begin{cases}
\begin{array}{l}
 \dfrac{\partial \omega^2(t)}{\partial t} = -3\sigma_0(t)^2\omega^2(t),\\[5pt]
\frac{f'(t)}{f(t)} \psi_+(t)+ \dfrac{\partial \psi_+(t)}{\partial t} =  -3\sigma_0(t)^2\psi_+(t).
 \end{array}
 \end{cases}
\end{equation}

Let us solve this system (as before, we denote $f_i(t)f_j(t)$ simply as $f_{ij}$).

\begin{lemma}\label{lemma1}
If $\dfrac{\partial \omega^2(t)}{\partial t} = -3\sigma_0(t)^2\,\omega^2(t)$,  then, $f_{12} = f_{34} = f_{56},$ where $f_i(t)$ are the functions in~\eqref{evolution}.
\end{lemma}

\begin{proof}
Using the symplectic operator $L:\Omega^q(M)\to \Omega^{q+2}(M)$ defined by $L(\eta)=\eta\wedge\omega$, the previous equation can be expressed as:
$$\dfrac{\partial \omega^2(t)}{\partial t} + 3\sigma_0(t)^2\,\omega^2(t)=0\Longleftrightarrow L_t\left(2\,\dfrac{\partial \omega(t)}{\partial t} +3\sigma_0(t)^2\,\omega(t) \right)=0.$$
It happens that $L$ is injective for $q\leq n-1$, being $\dim M=2n$ \cite{Bo}.  Since in our case $n=3$, we have that $$L_t\left(2\,\dfrac{\partial \omega(t)}{\partial t} + 3\sigma_0(t)^2\,\omega(t) \right)=0\Longleftrightarrow \dfrac{\partial \omega(t)}{\partial t}=-\frac32\sigma_0(t)^2\,\omega(t).$$

Using~\eqref{dw}, $\dfrac{\partial \omega(t)}{\partial t}=-\frac32\sigma_0(t)^2\,\omega(t)$ if and only if
$$\left(\frac{f_1'(t)}{f_1(t)} + \frac{f_2'(t)}{f_2(t)}\right)= \left(\frac{f_3'(t)}{f_3(t)} + \frac{f_4'(t)}{f_4(t)}\right) = \left(\frac{f_5'(t)}{f_5(t)} + \frac{f_6'(t)}{f_6(t)}\right)=-\frac32\sigma_0(t)^2,$$ which is equivalent to say
\begin{equation*}\label{una}
d(\ln f_{12})=d(\ln f_{34}) = d(\ln f_{56})=-\frac32\sigma_0(t)^2.
\end{equation*}
  In particular, $$\frac{f_{12}}{f_{34}} = c_1,\quad \frac{f_{12}}{f_{56}} = c_2,\quad \frac{f_{34}}{f_{56}} = c_3,$$ where $c_i$ are
  constants.  Since $f_i(0) = 1$, we obtain that $f_{12} = f_{34} = f_{56}.$
\end{proof}
\medskip

For the second equation we get:

\begin{lemma}\label{lemma2}
If $\frac{f'(t)}{f(t)} \psi_+(t)+ \dfrac{\partial \psi_+(t)}{\partial t} =  -3\sigma_0(t)^2\,\psi_+(t)$, then, $f_1(t) = f_2(t),\,\,f_3(t) = f_4(t),\,\,f_5(t) = f_6(t),$ where $f_i(t)$ are the functions in~\eqref{evolution}.
\end{lemma}

\begin{proof}
Arguing as before, if $\frac{f'(t)}{f(t)} \psi_+(t)+ \dfrac{\partial \psi_+(t)}{\partial t} =  -3\sigma_0(t)^2\,\psi_+(t)$, then:
\begin{equation*}\label{dos}
d(\ln (f(t) f_{135})) = d(\ln (f(t) f_{146})) = d(\ln (f(t) f_{236})) = d(\ln (f(t) f_{245})) = -3\sigma_0(t)^2.
\end{equation*}
In particular, observe that:
$$d(\ln (f(t) f_{ijk})) = d(\ln (f(t) f_{ipq}))\Longleftrightarrow d\left(\ln \frac{f(t) f_{ijk}}{f(t) f_{ipq}}\right) = 0\Longleftrightarrow \ln \frac{f_{jk}}{f_{pq}} = c \Longleftrightarrow  \frac{f_{jk}}{f_{pq}} = 1,$$
where $c$ is a constant and we have used the fact that $f_i(0)=1$.  So:
$$d(\ln (f(t) f_{135})) = d(\ln (f(t) f_{146})) = d(\ln (f(t) f_{236})) = d(\ln (f(t) f_{245})) \Longleftrightarrow$$
$$\begin{cases}
f_{13} = f_{24},\quad f_{14}=f_{23},\quad f_{15} = f_{26},\\
f_{16} = f_{25},\quad f_{35}=f_{46},\quad f_{36} = f_{45},
\end{cases}$$
$$
\Longleftrightarrow f_1(t)^2 = f_2(t)^2,\,\,f_3(t)^2 = f_4(t)^2,\,\,f_5(t)^2 = f_6(t)^2 \Longleftrightarrow f_1(t) = f_2(t),\,\,f_3(t) = f_4(t),\,\,f_5(t) = f_6(t),
$$ where for the last equivalence we have used that $f_i(t)$ are continuous functions satisfying $f_i(0)=1$.
\end{proof}

We can combine the two previous results to conclude that $f_i(t) = f_j(t)$ for $i,j=1,\ldots, 6$.  If we denote
$f_i(t)=F(t)$
for all $i=1,\dots,6$, then $(\omega(t),\psi_{\pm}(t))$ has the particular form:

\begin{equation}\label{equacion-NK5}
\omega(t)=F^2(t)\, \omega  ,\qquad
\psi_+(t)=F^3(t)\, \psi_+,\qquad
\psi_-(t)=F^3(t)\, \psi_-.
 \end{equation}

\begin{lemma}\label{lemma3}
Let $(\omega(t),\psi_{\pm}(t))$ be the one-parameter family of $\mathrm{SU}(3)$-structures given in \eqref{equacion-NK5} where $(\omega, \psi_{\pm})$ is a nearly K\"ahler structure. Then $(\omega(t),\psi_{\pm}(t))$
is nearly K\"ahler for all $t$ if and only if $\sigma_0(t)=\frac{\sigma_0}{F(t)}$.
\begin{proof}
Equation \eqref{equacion-NK5} implies that
$d\omega(t) = F^2(t) d\omega,$ and $d\psi_-(t)=F^3(t) d\psi_-.$
Since $(\omega, \psi_{\pm})$ is nearly K\"ahler, one has
$$d\omega(t)=-\frac{3}{2} \sigma_0 F^2(t) \psi_+, \quad \text{ } \quad d\psi_-(t)= \sigma_0 F^3(t) \omega^2,$$
or equivalently
$$d\omega(t)=-\frac{3}{2} \frac{\sigma_0}{ F(t)} \psi_+(t) \quad \text{ and } \quad d\psi_-(t)= \frac{\sigma_0}{ F(t)} \omega^2(t).$$
Therefore, $\big(\omega(t), \psi_{\pm}(t)\big)$ is nearly K\"ahler for all $t$ if and only if
$\sigma_0(t)=\frac{\sigma_0}{F(t)},$
and the result follows.
\end{proof}
\end{lemma}

In the next result we show how to solve the Laplacian coflow in this particular case.
\begin{proposition}\label{propoNK}Let $M^6$ be a manifold endowed with a nearly K\"ahler structure $(\omega, \psi_{\pm})$.
  Then the one-parameter family of warped $\mathrm{G}_2$-structures on $M^6 \times_f S^1$ given by
$$\varphi(t)=\left(1-\frac{3\sigma_0^2}{2} \, t\right)^{3/2}\left(c\, \omega\wedge ds - \psi_-\right)\quad \text{and} \quad \ast_t\varphi(t)=\left(1-\frac{3\sigma_0^2}{2} \, t\right)^{2}\left(\frac12\omega^2 +c \psi_+\wedge ds\right)$$
is a solution of the Laplacian coflow for $t\in\left(-\infty, \frac{2}{3\sigma^2_0}\right)$, being $f(t)=c\Big(1-\frac{3\sigma_0^2}{2} \, t\Big)^{1/2},\,\, c\in\mathbb R^*$.
\begin{proof}
From Lemmas \ref{lemma1}, \ref{lemma2} and \ref{lemma3}, the system \eqref{equacion-NK2} with $(\omega(t),\psi_{\pm}(t))$
nearly K\"ahler for all $t$ is equivalent to
\begin{equation*}\label{equacion-NK6}
 \begin{cases}
\begin{array}{l}
 4F'(t)F(t)=-3 \sigma_0^2,\\[5pt]
\frac{f'(t)}{f(t)}F^2(t)+3F'(t)F(t)=-3\sigma_0^2.
 \end{array}
 \end{cases}
 \end{equation*}
 whose solution is
 $$F(t)= \Big(1-\frac{3\sigma_0^2}{2} \, t\Big)^{1/2}, \text{ \quad  \quad } f(t)=c\,\Big(1-\frac{3\sigma_0^2}{2} \, t\Big)^{1/2}$$
 and the result follows.
\end{proof}
\end{proposition}

\begin{remark}
Not many examples of nearly K\"ahler manifolds are known.  Recently, new complete examples on $S^6$ and $S^3\times S^3$ have been described in \cite{FosHan} and \cite{Sc}.  Next we solve the Laplacian coflow using an explicit example of nearly K\"ahler structure appeared in \cite{Sc}.
\end{remark}

\begin{example}
 Consider the sphere $S^3$, viewed as the Lie group $\mathrm{SU}(2)$ with the basis of left-invariant one-forms $\{\lambda^1, \lambda^2, \lambda^3\}$ satisfying
$$d\lambda^1=\lambda^{23}, \quad d\lambda^2=-\lambda^{13}, \quad d\lambda^3=\lambda^{12}.$$
Thus, $\frak{su}(2)\oplus\frak{su}(2) $ is the Lie algebra of $S^3 \times S^3$ and its structure equations are:
$$\frak{su}(2)\oplus\frak{su}(2) = (\lambda^{23},-\lambda^{13},\lambda^{12}, \nu^{23}, -\nu^{13}, \nu^{12})$$
with $\{\nu^i\}$ the basis of left-invariant $1$-forms on the second sphere. The pair $(\omega, \psi_+)$ with
$$\omega=\frac{\sqrt{3}}{18}(\lambda^1\wedge \nu^1+\lambda^2\wedge \nu^2+\lambda^3\wedge \nu^3),$$ $$
\psi_+=\frac{\sqrt{3}}{54}(\lambda^{23}\wedge \nu^1-\lambda^1\wedge \nu^{23}-\lambda^{13}\wedge \nu^2+\lambda^{2}\wedge\nu^{13}+\lambda^{12}\wedge\nu^{3}-\lambda^{3}\wedge\nu^{12}),$$
where $\omega$ is the K\"ahler form and $\psi_+$ is the real part of the complex (3,0)-form, defines a nearly K\"ahler SU(3)-structure on $S^3 \times S^3$. Observe that the basis $\{\lambda^i,\, \nu^i\}$ is not adapted to the SU(3)-structure.

Consider $\{h^1, \dots, h^6\}$ the basis of left-invariant $1$-forms on $S^3 \times S^3$ given by
$$h^1=\frac13\lambda^1 - \frac16 \nu^1,\quad h^2=\frac{\sqrt{3}}{6}\nu^1,\quad h^3=\frac13\lambda^2 - \frac16 \nu^2,\quad h^4=\frac{\sqrt{3}}{6}\nu^2,\quad  h^5=\frac{\sqrt{3}}{6}\nu^3,\quad h^6=-\frac13\lambda^3 + \frac16 \nu^3.$$
This basis is adapted to the SU(3)-structure and $(\omega, \psi_+)$ turns out to be nearly K\"ahler with $\sigma_0 = -2$.
Therefore, in view of Proposition \ref{propoNK}, the one-parameter family of warped $\mathrm{G}_2$-structures on $(S^3\times S^3) \times_f S^1$ given by
$$\varphi(t)=\Big(1-6 \, t\Big)^{3/2}\Big[c (h^{12}+h^{34}+h^{56})\wedge ds + h^{246} - h^{235} - h^{136} -h^{145}  \Big]$$
and
$$\ast_t\varphi(t)=\Big(1-6 \, t\Big)^{2}\Big[h^{1234}+h^{1256}+h^{3456}+c (h^{135}-h^{146}-h^{236}-h^{245})\wedge ds   \Big],$$
where $f(t) = c\,\left(1-6t\right)^{\frac12}$, is a solution of the Laplacian coflow for all $t \in \big(-\infty, \frac{1}{6}\big)$.
\end{example}

\subsection{The symplectic half-flat case}
Recall that a symplectic half-flat SU(3)-structure satisfies 
\begin{equation}\label{su3-sH}
d\omega= 0,\quad d\psi_+=0,\quad d\psi_- = -\sigma_2\wedge\omega.
\end{equation}
 In particular, $\sigma_0=\nu_3=0$.
Particularizing~\eqref{equacion} for  $\sigma_0(t)=\nu_3(t)=0,$ we get
\begin{equation}\label{equacion-SHF}
  \begin{cases}
\begin{array}{l}
 \dfrac{\partial \omega^2(t)}{\partial t} = 0,\\[7pt]
\dfrac{f'(t)}{f(t)} \psi_+(t)+ \dfrac{\partial \psi_+(t)}{\partial t} =  d_6\sigma_2(t).
 \end{array}
 \end{cases}
\end{equation}

Now, we get necessary conditions in order to solve the Laplacian coflow.  Arguing similarly as Lemma~\ref{lemma1} and providing that  $\sigma_0(t)=0$, it is straightforward to see that  the first equation of~\eqref{equacion-SHF} holds if and only if
\begin{equation}\label{condition-f}
f_{2}(t) = \frac{1}{f_1(t)},\quad f_{4}(t) = \frac{1}{f_3(t)},\quad f_{6}(t) = \frac{1}{f_5(t)}.
\end{equation}
The following technical result, that makes use of equation~\eqref{segunda}, states how to solve the coflow in the symplectic half-flat case:

\begin{lemma}\label{cor:SHF}
Consider a warped coclosed $\mathrm{G}_2$-structure $\varphi$ on $M^6\times_f S^1$ where $(\omega, \psi_{\pm})$ is a symplectic half-flat $\mathrm{SU}(3)$-structure.  Then $\varphi(t)$, given by \eqref{varphi-t}, is a solution of the coflow~\eqref{equacion-SHF} using the ansatz~\eqref{evolution} if and only if $f(t),\, f_1(t),\, f_3(t)$ and $f_5(t)$ satisfy:
\begin{equation}\label{equacion-SHF2}
\begin{cases}
\begin{array}{l}
A_{135}(t) = \dfrac{f'(t)}{f(t)} + \dfrac{f_1'(t)}{f_1(t)} + \dfrac{f_3'(t)}{f_3(t)} + \dfrac{f_5'(t)}{f_5(t)},\qquad
A_{146}(t) = \dfrac{f'(t)}{f(t)} + \dfrac{f_1'(t)}{f_1(t)} - \dfrac{f_3'(t)}{f_3(t)} - \dfrac{f_5'(t)}{f_5(t)},\\[10pt]
A_{236}(t) = \dfrac{f'(t)}{f(t)} - \dfrac{f_1'(t)}{f_1(t)} + \dfrac{f_3'(t)}{f_3(t)} - \dfrac{f_5'(t)}{f_5(t)},\qquad
A_{245}(t) = \dfrac{f'(t)}{f(t)} - \dfrac{f_1'(t)}{f_1(t)} - \dfrac{f_3'(t)}{f_3(t)} + \dfrac{f_5'(t)}{f_5(t)},\\
\end{array}
\end{cases}
\end{equation}
where functions $A_{135}(t), A_{146}(t), A_{236}(t), A_{245}(t)$ are such that $$d_6\sigma_2(t)= A_{135}(t) x^{135} - A_{146}(t) x^{146} - A_{236}(t) x^{236} - A_{245}(t)x^{245},$$ and $(\omega(t),\, \psi_{\pm}(t))$ is symplectic half-flat for all $t$.
 \end{lemma}

In order to obtain examples and inspired in the solutions given in Proposition~\ref{propoNK}, we will consider the functions $f_i(t)$ of potential type, i.e.
\begin{equation}\label{f-especial}
f_i(t) = (1+kt)^{\alpha_i}
\end{equation}
with $\alpha_i$ and $k$ real numbers. Thus the solutions of the coflow are of the form:
\begin{equation}\label{varphialphabeta}
\begin{aligned}
\varphi(t) &= f(t)\left[(1+k t)^{\alpha_1 + \alpha_2 } h^{12}+ (1+k t)^{\alpha_3 + \alpha_4} h^{34} +(1+k t)^{\alpha_5 + \alpha_6 } h^{56}\right]\wedge ds \\
 &- (1+k t)^{\alpha_2 + \alpha_4 + \alpha_6} h^{246} +(1+k t)^{\alpha_2 + \alpha_3 + \alpha_5} h^{235} +(1+k t)^{\alpha_1 + \alpha_4 + \alpha_5} h^{145} +(1+k t)^{\alpha_1 + \alpha_3 + \alpha_6} h^{136},
\end{aligned}
\end{equation}
where the basis $\{h^1,\,\ldots,\,h^6\}$ is defined in \eqref{evolution}.

Next we solve the Laplacian coflow on unimodular solvable Lie algebras.

\begin{example}
Consider the Lie algebra $\frak{e}(1,1)\oplus \frak{e}(1,1)$ whose structure equations are
\begin{equation*}
\frak{e}(1,1)\oplus \frak{e}(1,1) := (0,0, -h^{14}, -h^{13}, h^{25}, -h^{26}).
\end{equation*}
The corresponding  connected and simply connected Lie group $G$  admits a left-invariant symplectic half-flat structure which is given canonically by~\eqref{adaptedBasis} in basis $\{h^i\}$.
Let us consider a one-parameter family of $\mathrm{SU}(3)$-structures given by~\eqref{SU3-x} with  
 $x^i(t)=f_i(t)h^i$ being $f_i(t)$ of potential type as in~\eqref{f-especial}.  The structure equations of $\frak{e}(1,1)\oplus \frak{e}(1,1)$ with respect to the time-dependent basis $\{x^i(t)\}$ are
\begin{equation*}
(0,0, -(1+kt)^{\alpha_3-\alpha_1-\alpha_4}x^{14}, -(1+kt)^{\alpha_4-\alpha_1-\alpha_3}x^{13}, (1+kt)^{-\alpha_2}x^{25}, -(1+kt)^{-\alpha_2}x^{26}).
\end{equation*}
In order to obtain solutions for the Laplacian coflow, and in view of \eqref{condition-f}, we can set
\begin{equation*}
\alpha_2=-\alpha_1, \qquad \alpha_4=-\alpha_3, \quad \text{ and } \quad \alpha_6=-\alpha_5.
\end{equation*}

With these values, we impose the preservation of the symplectic half-flat condition.   It is easy to verify that $d\omega(t)=0$ for all $t$;  $\psi_+(t)$ remains closed if and only if $\alpha_1=\alpha_3=0$, since 
$$d\psi_+(t)= \, \big(-(1+kt)^{\alpha_1}+(1+kt)^{-\alpha_1-2\alpha_3}\big)x^{1235}+\big(-(1+kt)^{\alpha_1}+(1+kt)^{-\alpha_1+2\alpha_3}\big)x^{1246}.$$   So, $(\omega(t), \psi_{\pm}(t))$ is symplectic half-flat for all $t$ if and only if $\alpha_1=\alpha_2=\alpha_3=\alpha_4=0.$
Observe that the structure equations are simply:
\begin{equation*}
\frak{e}(1,1)\oplus \frak{e}(1,1) := (0,0, -x^{14},-x^{13}, x^{25}, -x^{26}).
\end{equation*}

Finally, to solve the second equation of \eqref{equacion-SHF} we make use of~\eqref{equacion-SHF2}. 
Since $(\omega(t),\psi_{\pm}(t))$ is symplectic half-flat for all $t$, $\sigma_2(t)=-\ast_t d \psi_-(t)$, see~\eqref{SU3-str}, and therefore
\begin{equation*}
d\sigma_2(t)=-2x^{135}+2x^{146}+2x^{236}+2x^{245},
\end{equation*} which means that $A_{ijk}(t) = -2$.
We obtain the system
\begin{equation*}
\begin{cases}
\frac{f'(t)}{f(t)}+k\alpha_5(1+kt)^{-1}=-2, \\[5pt]
\frac{f'(t)}{f(t)}-k\alpha_5(1+kt)^{-1}=-2.
 \end{cases}
\end{equation*}
%
which can be solved taking
\begin{equation*}
\alpha_5=0 \qquad  \text{ and } \qquad f(t)=c\,e^{-2t},\quad c\in\mathbb R^*.
\end{equation*}
Therefore, the one-parameter family of $\mathrm{G}_2$-structures on $G \times_f S^1$ given by~\eqref{varphialphabeta}
\begin{equation*}
\varphi(t)=c\,e^{-2t}(h^{12}+h^{34}+h^{56})\wedge ds - h^{246} +h^{235} +h^{145} +h^{136}
\end{equation*}
is a solution of the Laplacian coflow for all $t \in \mathbb{R}$.

\end{example}

In~\cite{FMOU}, the authors classify the 6-dimensional  unimodular solvable Lie algebras admitting symplectic half-flat $\mathrm{SU}(3)$-structure  and show that all the corresponding solvable Lie groups admit a co-compact discrete
subgroup.  In addition to the Lie algebra $\frak e(1,1)\oplus \frak e(1,1)$, in terms of an adapted basis $\{h^i\}_{i=1}^6$ to the $\mathrm{SU}(3)$-structure, the structure equations of these algebras are the following:
\begin{eqnarray*}
\frg_{5,1}\oplus\mathbb R&=&(0,0,0, h^{15}, 0,h^{13}),\\
A^{-1,-1,1}_{5,7}\oplus\mathbb R&=&(h^{16}, -h^{26}, -h^{36}, h^{46}, 0,0),\\
A^{-a,-a,1}_{5,17}\oplus\mathbb R&=&(a h^{15}+h^{35}, -a h^{25}+h^{45}, -h^{15}+a h^{35}, -h^{25}-a h^{45},0,0),\\
\frg_{6, N3}&=&(0,-2 h^{35},0, -h^{15}, 0,h^{13}),\\
\frg_{6,38}^0&=&(2 h^{36},0,-h^{26}, h^{25}-h^{26}, -h^{23}-h^{24}, h^{23}),\\
\frg_{6, 54}^{0,-1}&=&\left(\dfrac{h^{16}}{\sqrt2}+h^{45}, -\dfrac{h^{26}}{\sqrt2}, h^{25}-\dfrac{h^{36}}{\sqrt2}, \dfrac{h^{46}}{\sqrt2},0,0\right),\\
\frg_{6, 118}^{0, -1, -1}&=&(-h^{15}+h^{36}, h^{25}+h^{46}, -h^{16}-h^{35},  -h^{26}+h^{45},0,0).
\end{eqnarray*}

In Table~\ref{tabla-SHF} we present long time solutions to the Laplacian coflow for $\mathrm{G}_2$-structures obtained as warped products of
solvmanifolds endowed with symplectic half-flat $\mathrm{SU}(3)$-structures. These solutions can be obtained as follows: consider
Lemma~\ref{cor:SHF} with the potential functions given in \eqref{f-especial} and a warping function also of potential type $$f(t) =c\, (1+kt)^{\beta},\quad c\in\mathbb R^*.$$ Thus, using~\eqref{equacion-SHF2}, we obtain a linear system of equations in $\alpha_i, \beta$ and $k$ that can be easily solved. Known the values of  $\alpha_i, \beta$ and $k$ and
considering \eqref{varphialphabeta} we can give an explicit description of the solutions of the Laplacian coflow for each example.  We also include the value of $d\sigma_2(t)$ in each case, necessary to compute the parameters of the solutions. 
%

\begin{table}[h!]
\renewcommand{\arraystretch}{1.5}
\begin{center}
\begin{tabular}{|c|c||c|c|c|}
\hline
Lie algebra&$d\sigma_2(t)$& $(\alpha_1,\,\ldots,\,\alpha_6)$ & $\beta$ & $k$  \\
\hline
$\frg_{5,1}\oplus\mathbb R$ &$A_{135} = -2(1+kt)^{-2\alpha_1-2\alpha_3-2\alpha_5}$&$( \frac16, -\frac16, \frac16, -\frac16, \frac16, -\frac16)$ & $\frac16$ & $-3$\\ \hline
$A_{5,7}^{-1,-1,1}\oplus\mathbb R$ &$A_{146} =A_{236}= -4(1+kt)^{2\alpha_5}$&$( 0,0,0,0, -\frac12, \frac12)$ & $\frac12$ & $-4$\\ \hline
$A_{5,17}^{-a,-a,1}\oplus\mathbb R$ &$A_{135} =A_{245}= -4a^2(1+kt)^{-2\alpha_5}$&$( 0,0,0,0, \frac12, -\frac12)$ & $\frac12$ & $-4a^2$\\ \hline
$\frg_{6,N3}$ &$A_{135} = -6(1+kt)^{-2\alpha_1-2\alpha_3-2\alpha_5}$&$( \frac16, -\frac16, \frac16, -\frac16, \frac16, -\frac16)$ & $\frac16$ & $-9$\\ \hline
$\frg_{6,38}^0$ &$A_{236} = -6(1+kt)^{2\alpha_1-4\alpha_3}$&$( -\frac16, \frac16, \frac16, -\frac16, -\frac16, \frac16)$ & $\frac16$ & $-9$\\ \hline
\multirow{2}{*}{$\frg_{6,54}^{0,-1}$} &$A_{146} =A_{236}= -2(1+kt)^{2\alpha_5}$&\multirow{2}{*}{$( -\frac12, \frac12, -\frac12, \frac12, -\frac12, \frac12)$} & \multirow{2}{*}{$\frac32$} & \multirow{2}{*}{$-1$}\\
&$A_{245} = -2(1+kt)^{2\alpha_1+2\alpha_3-2\alpha_5}$&&&\\
\hline
\multirow{2}{*}{$\frg_{6,118}^{0,-1,-1}$} &$A_{135} = A_{245} = -4(1+kt)^{-2\alpha_5}$&\multirow{2}{*}{$( 0,0,0,0, \frac12, -\frac12)$} & \multirow{2}{*}{$\frac12$} & \multirow{2}{*}{$-4$}\\
&$A_{146} = A_{236} = $&&&\\
&$-2(1+kt)^{2\alpha_5} (-1+(1+kt)^{2\alpha_1-2\alpha_3})$&&&\\
\hline
\end{tabular}
\end{center}
\caption{Solutions of the Laplacian coflow in the SHF-case}
\label{tabla-SHF}
\end{table}

\subsection{The balanced case}

Recall that a balanced SU(3)-structure satisfies 
\begin{equation}\label{su3-bA}
d\omega= \nu_3,\quad d\psi_+=0,\quad d\psi_- = 0.
\end{equation}
 In particular, $\sigma_0=\sigma_2=0$.
Particularizing~\eqref{equacion} for  $\sigma_0(t)=\sigma_2(t)=0,$ we get
\begin{equation}\label{equacion-Bal}
 \begin{cases}
\begin{array}{l}
\dfrac{\partial \omega^2(t)}{\partial t} = 2d_6(\ast_6\nu_3(t)),\\[7pt]
\dfrac{f'(t)}{f(t)} \psi_+(t)+ \dfrac{\partial \psi_+(t)}{\partial t} =  0.
\end{array}
\end{cases}
\end{equation}

In this case, we can apply Lemma~\ref{lemma2} with $\sigma_0(t)=0$ (compare the second equations in~\eqref{equacion-NK2} and \eqref{equacion-Bal}) obtaining the same conclusion, i.e., $f_{2k}(t) = f_{2k-1}(t)$ for $k=1,2,3$.  Now, similarly to Lemma~\ref{cor:SHF}, we can set:

\begin{lemma}\label{cor:balanced}
Consider a warped coclosed $\mathrm{G}_2$-structure $\varphi$ on $M^6\times_f S^1$ where $(\omega, \psi_{\pm})$ is a balanced $\mathrm{SU}(3)$-structure.  Then $\varphi(t)$, given by \eqref{varphi-t}, is a solution of the coflow~\eqref{equacion-Bal} using the ansatz~\eqref{evolution} if and only if $f(t),\, f_1(t),\, f_3(t)$ and $f_5(t)$ satisfy:
\begin{equation*}
B_{1234}(t) = 2\left( \dfrac{f_1'(t)}{f_1(t)} + \dfrac{f_3'(t)}{f_3(t)}\right), \quad
B_{1256}(t) = 2\left( \dfrac{f_1'(t)}{f_1(t)} + \dfrac{f_5'(t)}{f_5(t)}\right),\quad
B_{3456}(t) = 2\left( \dfrac{f_3'(t)}{f_3(t)} + \dfrac{f_5'(t)}{f_5(t)}\right),
\end{equation*}
where functions $B_{1234}(t),\,B_{1256} (t),\,B_{3456}(t)$ are such that $$d_6(\ast \nu_3(t) )= B_{1234}(t) x^{1234} + B_{1256}(t) x^{1256} + B_{3456}(t) x^{3456},$$ and $(\omega(t),\, \psi_{\pm}(t))$ is balanced for all $t$.
\end{lemma}

\medskip

The examples that we present in this case are the 6-dimensional nilpotent Lie algebras admitting balanced $\mathrm{SU}(3)$-structures, that are classified in~\cite{UV}.  In terms of an adapted basis to the balanced $\mathrm{SU}(3)$-structure, the structure equations are:
\begin{eqnarray*}
\frak h_2&=&(0,0,0,0,2 h^{12}+\left(2 \sqrt{2}-2\right) h^{13}+\left(-2-2 \sqrt{2}\right) h^{24}-2 h^{34}, 4 \sqrt{2} h^{12}+4 \sqrt{2} h^{23}-4 \sqrt{2} h^{34}),\\
\frak h_3&=&(0,0,0,0,0,-2h^{12}+2h^{34}),\\
\frak h_4&=&(0,0,0,0,2h^{13},h^{14}+h^{23}),\\
\frak h_5&=&(0,0,0,0, h^{13}-h^{24}, h^{14}+h^{23}),\\
\frak h_6&=&(0,0,0,0, h^{13},h^{14}),\\
\frak h_{19}^-&=&(0,0,-h^{15},-h^{25},0, -h^{13}-h^{24}).
\end{eqnarray*}

We present long time solutions for the Laplacian coflow of $\mathrm{G}_2$-structures obtained as warped products of balanced nilmanifolds endowed with $\mathrm{SU}(3)$-structures.  These solutions remain balanced for any $t$. As before,  with the notation in Lemma~\ref{cor:balanced} and functions of potential type~\eqref{f-especial} giving an explicit description of these solutions is equivalent to obtain the values of the parameters $\alpha_i, \beta$ and $k$. Solving the corresponding linear equations these values are given in Table~\ref{tabla-balanced}.  The solutions $\varphi(t)$ of the coflow are of the form~\eqref{varphialphabeta}.  We also include the value of $d\ast \nu_3(t)$ in each case, necessary to compute the parameters of the solutions.

\begin{table}[h!]
\renewcommand{\arraystretch}{1.5}
\begin{center}
\begin{tabular}{|c|c||c|c|c|}
\hline
Lie algebra&$d\ast \nu_3(t)$& $(\alpha_1,\,\ldots,\,\alpha_6)$ & $\beta$ & $k$  \\
\hline
$\frh_2$&$B_{1234} = -128(1+kt)^{-4\alpha_1+2\alpha_5}$&$( \frac16, \frac16, \frac16, \frac16, -\frac16, -\frac16)$ & $-\frac16$ & $-192$\\
\hline
$\frh_3$ &$B_{1234} = -8(1+kt)^{-4\alpha_1+2\alpha_5}$&$( \frac16, \frac16, \frac16, \frac16, -\frac16, -\frac16)$ & $-\frac16$ & $-12$\\ \hline
$\frh_4$ &$B_{1234} = -6(1+kt)^{-2\alpha_1-2\alpha_3 + 2\alpha_5}$&$( \frac16, \frac16, \frac16, \frac16, -\frac16, -\frac16)$ & $-\frac16$ & $-9$\\ \hline
$\frh_5$ &$B_{1234} = -4(1+kt)^{-2\alpha_1-2\alpha_3+2\alpha_5}$&$( \frac16, \frac16, \frac16, \frac16, -\frac16, -\frac16)$ & $-\frac16$ & $-6$\\ \hline
$\frh_6$ &$B_{1234} = -2(1+kt)^{-2\alpha_1-2\alpha_3+2\alpha_5}$&$( \frac16, \frac16, \frac16, \frac16, -\frac16, -\frac16)$ & $-\frac16$ & $-3$\\ \hline
\multirow{2}{*}{$\frh_{19}^-$} &$B_{1234} = -2(1+kt)^{-2\alpha_1-2\alpha_3+2\alpha_5}$&\multirow{2}{*}{$( \frac12, \frac12, 0,0,0,0)$} & \multirow{2}{*}{$-\frac12$} & \multirow{2}{*}{$-2$}\\
&$B_{1256} = -2(1+kt)^{-2\alpha_1+2\alpha_3-2\alpha_5}$&&&\\
\hline
\end{tabular}
\end{center}
\caption{Solutions of the Laplacian coflow in the balanced case}
\label{tabla-balanced}
\end{table}

\section*{Acknowledgments}
\noindent
The authors would like to thank Anna Fino, Jorge Lauret and Luis Ugarte for useful comments on the subject.  
 This work has been partially supported by the projects MTM2017-85649-P (AEI/Feder, UE), and E22-17R ``\'Algebra y Geometr\'{\i}a'' (Gobierno de Arag\'on/FEDER)


\end{document}